\documentclass [11pt] {amsart}
\usepackage{amsmath, latexsym, amssymb, graphicx,amsthm}
\usepackage[left=1.4in,top=1in,right=1.4in]{geometry}
\newtheorem{thm}{Theorem}

\newtheorem{lem}[thm]{Lemma}

\theoremstyle{remark}

\newcommand{\R}{\ensuremath{\mathbb{R}} }

\newcommand{\ZmZ}{ \ensuremath{\mathbb{Z}/m \mathbb{Z}}}
\newcommand{\ZpZ}{ \ensuremath{\mathbb{Z}/p \mathbb{Z}}}
\newcommand{\Fptwo}{ \ensuremath{\mathbb{F}_{p^2}}}
\newcommand{\Fp}{ \ensuremath{\mathbb{F}_p}}

\newcommand{\U}{U}

\begin{document}

\title[Periods of Fibonacci sequences mod primes]{Splitting fields and periods of Fibonacci sequences modulo primes}

\author[Gupta]{Sanjai Gupta}
\address{Department of Mathematics\\ Irvine Valley College\\ 
Irvine, CA 92618}
\email{sgupta@ivc.edu}

\author[Rockstroh]{Parousia Rockstroh}
\address{Department of Mathematics\\ Simon Fraser University\\                                                  
Burnaby, British ColumbiaÊÊV5A 1S6}
\email{parousia\_rockstroh@sfu.ca}

\author[Su]{Francis Edward Su$^*$}
\address{Department of Mathematics\\ Harvey Mudd College\\
Claremont, CA 91711}
\email{su@math.hmc.edu}


\thanks{$^*$Supported in part by NSF DMS-0701308.}


\maketitle

\section{Introduction}
The Fibonacci sequence defined by $F_0=0, F_1=1, F_{n+1}=F_n+F_{n-1}$
is clearly periodic when reduced modulo an integer $m$, since there
are only finitely many possible pairs of consecutive elements chosen from
$\ZmZ$ (in fact, $m^2$ such pairs) 
and any such pair determines the rest of the sequence.  What is the
period of this sequence?  

An upper bound is $m^2-1$ (since the sequence does not have a
consecutive pair of $0$'s), but the period is often much smaller.
As examples, the Fibonacci sequence mod $11$ is:
$$
0,1,1,2,3,5,8,2,10,1,0,1,1, \ldots
$$
and has period $10$; the Fibonacci sequence mod $7$ is:
$$
0,1,1,2,3,5,1,6,0,6,6,5,4,2,6,1,0,1,1,\ldots
$$
and has period $16$.


This problem was first considered by Wall \cite{Wall} and shortly
thereafter by Robinson \cite{Robinson}.  Among other cases, they
studied the Fibonacci sequence for prime moduli, and 
showed that for primes $p \equiv 1, 4$ mod $5$ the
period length of the Fibonacci sequence mod $p$ 
divides $p-1$, while for primes 
$p \equiv 2, 3$ mod $5$, the period length divides $2(p+1)$. 
The examples above illustrate these facts.

Wall's proofs use different combinatorial techniques for each of
these classes of primes.  
Robinson proves these results by appealing to a directed
graph of points formed by multiplication by a {\em Fibonacci matrix}.
In this paper, we give alternative proofs of these
results that also use the Fibonacci matrix, but unlike Robinson, 
we place the roots of its characteristic polynomial 
in an appropriate splitting field.  This allows us to 
obtain bounds for the periods of the more general recurrence
$$
E_{n+1}= A E_n + B E_{n-1}
$$
modulo a prime, which neither Wall nor Robinson consider.

Vella and Vella \cite{Vella-Vella} consider general recurrences, but only in the special case      
where the roots of characteristic polynomial are integers (viewed as a
polynomial with real coefficients).
Using sophisticated methods, Pinch \cite{Pinch}
proves general results about multiple-term recurrences with prime power  
moduli, but does not produce specific bounds of the kind that we consider here.
Li \cite{Li} reviews prior work on period lengths of general recurrences in 
the context of a different problem: determining which residue classes appear in recurrence sequences.

The purpose of our brief paper is to illustrate an accessible, 
motivated treatment of this classical topic using only ideas from linear and abstract algebra
(rather than the case-by-case analysis found in many papers on the subject, or techniques from graduate number theory).
Our methods extend to general recurrences with prime moduli and provide some new insights, e.g., Theorem \ref{Bsquaredbound}.
And our treatment highlights a nice application of the use of splitting fields 
that might be suitable to present in undergraduate course in abstract algebra or Galois theory.

\section{Eigenvalues of the Fibonacci Matrix}

Let $p$ be an odd prime.

In accordance with previous literature \cite{Robinson,Wall} we define
$k(p)$, the {\em period} of the Fibonacci sequence mod $p$, 
to be the smallest non-zero index $i$ such that $F_i\equiv 0$ mod $p$ 
and $F_{i+1}\equiv 1$ mod $p$.  
In our examples above,
$k(11)=10$, while $k(7)=16$.
Following Robinson \cite{Robinson}, we consider the Fibonacci matrix:
$$
\U=\left[
 \begin{array}{cc} 
	1 & 1\\
	1 & 0
\end{array}
\right]. 
$$
This is a matrix over some field $F$ that we should be careful to
specify. If we choose $F=\R$, then 
$$
\U^n=
\left[ \begin{array}{cc} 
	F_{n+1} & F_n\\
	F_n & F_{n-1}
\end{array} \right].
$$
And if we choose $F=\ZpZ$ (also written $\Fp$, the finite field of order $p$)
then the entries of $\U^n$ are elements of
the Fibonacci sequence mod $p$, the desired objects of study.

It is natural to consider the eigenvalues of the matrix $\U$,
which are roots of its characteristic polynomial $x^2-x-1$.  
If eigenvalues $\lambda, \bar{\lambda}$ exist and are distinct, 
then $\U=C D C^{-1}$ where $D$ is the diagonal matrix
$$
D=
\left[ \begin{array}{cc} 
	\lambda & 0\\
	     0  & \bar{\lambda}
\end{array} \right]
$$
and $C$ is a matrix with the corresponding eigenvectors as columns.
(If the eigenvalues are not distinct, then $D$ is not diagonal but
a Jordan block and $C$ is a matrix of generalized eigenvectors.)
Then $\U^k=C D^k C^{-1}$.  We see that for $k=k(p)$,
we have $U^k=I$, the
identity matrix.  Therefore $D^k= C^{-1} \U^k C$ is also
$I$.  We observe that the exponent $k=k(p)$ is the 
smallest non-zero exponent $n$ such that $D^n=I$.  Thus:

\begin{lem}
\label{cycle-criterion}
The period $k(p)$ must divide any $n$ that satisfies $D^n=I$.
\end{lem}

When do the eigenvalues $\lambda, \bar{\lambda}$ exist?  
The quadratic formula
shows that $ax^2+bx+c$ has roots in the field $\ZpZ$ as long as the
discriminant $\Delta = b^2-4ac$ is a square in $\ZpZ$; 
hence the characteristic polynomial 
$x^2-x-1$ has roots in $\ZpZ$ if and only if $\Delta=5$ is a square.
Quadratic reciprocity shows that if $p$ is an odd prime,
then $5$ is a square in $\ZpZ$ if and only if 
$p\equiv 0,1,4$ mod $5$.  And as long as $p\neq 5$, the eigenvalues are
distinct, hence we recover the classical bound:

\begin{thm}
\label{thm:14mod5}
If $p$ is an odd prime and $p \equiv 1,4$ mod $5$, then $k(p)$
divides $p-1$.  In particular, $k(p)\leq p-1$.
\end{thm}


\begin{proof} 
The eigenvalues $\lambda, \bar{\lambda}$ 
of $\U$ are non-zero (since $\U$ is invertible) and distinct (since
$p\neq 5$).
Since $p$ is prime, Fermat's (little) theorem implies both $\lambda^{p-1}=1$
and $\bar{\lambda}^{p-1}=1$.  Hence $D^{p-1}=I$ and Lemma
\ref{cycle-criterion} gives the desired conclusion.
\end{proof}

When $p=5$, the eigenvalues are not distinct and $D$ is not diagonal, so
$D^4 \neq I$ even though $\lambda^4=\bar{\lambda}^4=1$.  One finds
that $D^{20}=I$ and $k(5)=20$.

\section{A Splitting Field for the Eigenvalues}

A slight modification will take care of the remaining classes of primes 
$p \equiv 2,3$ mod $5$; 
but for such $p$ the characteristic polynomial
$x^2-x-1$ will not have roots in $\Fp$ unless we enlarge the field.  

In this case, we choose $F=\Fp [x]/(x^2-x-1)$, the splitting field of
$x^2-x-1$, and consider $\U$ as a matrix with entries in $F$.  
Note that $F$ 
is isomorphic to $\Fptwo$, the finite field of order $p^2$.  It has
$\Fp$ as a subfield, namely the image of the constant polynomials in $\Fp[x]$.
The quadratic formula gives the eigenvalues of $\U$:
\begin{equation}
\label{e-values}
\lambda=2^{-1}(1+\sqrt{5}), \qquad   
\bar{\lambda}=2^{-1}(1-\sqrt{5})
\end{equation}
where $\sqrt{5}$ denotes a field element of $F$ whose square is 5
(there are two; fix one).  This element has a special property:

\begin{lem}
\label{GaloisFact}
If $p\equiv 2,3$ mod $5$, then $(\sqrt{5})^p=-\sqrt{5}$.
\end{lem}
\begin{proof}
Consider the {\em Frobenius map} $\sigma: \Fptwo \rightarrow \Fptwo$ 
where $\sigma(\alpha)=\alpha^p$. It is well-established 
\cite{Dummit-Foote} that the Frobenius map is an automorphism of
$\Fptwo$ that fixes $\Fp$, hence it must permute the roots
of irreducible polynomials with coefficients in $\Fp$.

In particular, $\sigma$ permutes the roots of $x^2-5$, so
either $\sigma(\sqrt{5})=\sqrt{5}$ or $\sigma(\sqrt{5})=-\sqrt{5}$, 
i.e., either $\sigma$ fixes the entire
field $\Fptwo$ or just the subfield $\Fp$.  
But $\sigma$ is not the
identity, since the multiplicative group of $\Fptwo$ is known to be
cyclic \cite[p.314]{Dummit-Foote} of order $p^2-1$, 
so if the multiplicative generator is $\gamma$, then
$\sigma(\gamma)=\gamma^p \neq \gamma$.
Hence $\sigma(\sqrt{5})=-\sqrt{5}$, as desired.
\end{proof}

\begin{lem}
\label{Pmap}
Let $\lambda$ and $\bar{\lambda}$ be the two roots of $x^2-x-1$ in
$F=\Fptwo$. Then 
$$
\bar{\lambda}^{p+1}=\lambda^{p+1}.
$$
\end{lem}

\begin{proof}

We make frequent use of the following fact \cite[p.548]{Dummit-Foote}:
if $a,b\in \mathbb{F}_{p^2}$, then $(a+b)^p=a^p+b^p$. This follows
from the binomial theorem, noting that ${p \choose n}$ is divisible by $p$ 
if $p$ is prime and $n$ is not $0$ or $p$.

From (\ref{e-values}), note that $\bar{\lambda}=1-\lambda$.  Then 
$
\bar{\lambda}^{p+1}
=(1-\lambda)^{p+1}
=(1-\lambda)^p (1-\lambda)
=(1-\lambda^p) (1-\lambda)
=1-\lambda-\lambda^p+\lambda^{p+1}.
$
The desired result then follows from 
this claim: that $1-\lambda-\lambda^p=0$.  
Substituting (\ref{e-values}),
and using $(1+\sqrt{5})^p=1+\sqrt{5}^p$, we find:
$$
1-\lambda-\lambda^p =
1- 2^{-1}(1+\sqrt{5})-(2^{-1})^p (1+\sqrt{5}^p). 
$$
But since $2^{-1}$ is in $\Fp$, by Fermat's theorem, 
$\left(2^{-1}\right)^{p}=2^{-1}$, and the claim follows
from Lemma \ref{GaloisFact} and $1- 2 (2^{-1})=0$.
\end{proof}

Now we have enough to determine the desired bound:

\begin{thm}
\label{thm:23mod5}
Let $p$ be an odd prime with $p \equiv 2,3$ mod $5$ then 
$k(p)$ divides $2(p+1)$.  In particular, $k(p) \le 2(p+1)$.
\end{thm}

\begin{proof}
Recall that $\lambda$ and $\bar{\lambda}$ are roots of $x^2-x-1$.
Note that $\lambda\bar{\lambda}=-1$, and the above lemma shows that 
$\lambda^{p+1}=\bar{\lambda}^{p+1}$, so
$$
\lambda^{p+1} \lambda^{p+1} = \lambda^{p+1} \bar{\lambda}^{p+1}
= (-1)^{p+1}.
$$
Since $p$ is odd, $(-1)^{p+1} = 1$, so
$\lambda^{2(p+1)}=1$.  By Lemma
\ref{Pmap} we also have 
$\bar{\lambda}^{2(p+1)}=1$.
The result follows from Lemma \ref{cycle-criterion}.
\end{proof}

As Wall \cite{Wall} notes, the upper bounds of Theorems \ref{thm:14mod5} and
\ref{thm:23mod5} are tight for many small odd primes $p\neq 5$ (for
$p<100$, the only exceptions are 29, 47, and 89).  The bounds appear to be 
less tight for larger $p$.
Wall also shows for prime powers, $k(p^t) \leq p^{t-1}k(p)$ with
equality if $k(p^2)\neq k(p)$.  It is believed the latter condition
always holds; see \cite{Mamangakis} for partial results.
Combining knowledge of $k(p^t)$ with the fact
that $\mbox{lcm}[k(m),k(n)] = k(\mbox{lcm}[m,n])$, one can obtain a
bound on $k(m)$ for each positive integer $m$.

\section{The General Recurrence}

Our methods can be adapted to obtain bounds for the period of 
the general recurrence 
$$
E_{n+1}= A E_n + B E_{n-1}
$$
modulo a prime $p$, with $E_0=0$ and $E_1=1$.  
Let $k_{A,B}(p)$ be the period of $E_n$ mod $p$.  The Fibonacci matrix becomes 
$$
\U=\left[
 \begin{array}{cc} 
	A & B\\
	1 & 0
\end{array}
\right],
$$
and the eigenvalues $\lambda, \bar{\lambda}$ are roots of the 
characteristic polynomial $x^2-Ax-B$.  This has roots in $\ZpZ$ as long
as the discriminant $\Delta=A^2+4B$ is a square in $\ZpZ$ (a 
{\em quadratic residue} mod $p$), 
and they are distinct if $\Delta \not\equiv 0$ mod $p$.
The same arguments as in Theorem \ref{thm:14mod5} 
will yield:
\begin{thm}
If $p$ is an odd prime
and $\Delta$ is a non-zero quadratic residue mod $p$, then
$k_{A,B}(p)$ divides $p-1$.  In particular $k_{A,B}(p)\leq p-1$.
\end{thm}

For example, consider $E_{n+1}= 3 E_n + 2 E_{n-1} \mod 13$.  Then 
$A=3$, $B=2$, and $\Delta=17$.  Since $\Delta \equiv 2^2$ mod $13$, 
$\Delta$ is a non-zero quadratic residue mod $13$.  Our theorem shows that
$k_{3,2}(13) \leq 12$ (and, in fact, it is 12).

A curious consequence of our theorem is that the sequence 
$E_{n+1}=  E_n + 2 E_{n-1} \mod p$ has small period (that divides
$p-1$) for {\em every} odd prime $p$ except $3$ 
(since $\Delta=3^2$ is always a square, the only prime $p$ 
dividing $\Delta$ is $3$).

If the discriminant $\Delta$ is not a square in $\ZpZ$, we consider $U$ as
a matrix with entries from the splitting field of $x^2-Ax-B$, isomorphic
to $\Fptwo$ as before.  The proof of Lemma \ref{GaloisFact} holds
with $\sqrt{5}$ replaced by $\sqrt{\Delta}$ and noting that $\sigma$
permutes the roots of $x^2-\Delta$.  Thus:
\begin{lem}
If $\Delta$ is a quadratic nonresidue mod $p$, then 
$(\sqrt{\Delta})^p=-\sqrt{\Delta}$.
\end{lem}

The analogue of Lemma \ref{Pmap} still holds:
\begin{lem}
Let $\lambda$ and $\bar{\lambda}$ be the two roots of $x^2-Ax-B$ in
$F=\Fptwo$. Then
$$
\bar{\lambda}^{p+1}=\lambda^{p+1}.
$$
\end{lem}

This follows by a similar argument as in Lemma \ref{Pmap}, noting that 
$\bar{\lambda}=A-\lambda$, and 
$\bar{\lambda}^{p+1}=A(1-\lambda-\lambda^p) + \lambda^{p+1}$.
Thus it suffices to show, as before, that 
$1-\lambda-\lambda^p =0$. 
The same arguments hold, with $\sqrt{5}$ replaced by $\sqrt{\Delta}$. 

\begin{thm}
\label{Bsquaredbound}
If $\Delta$ is a quadratic nonresidue mod $p$, then
$k_{A,B}(p)$ divides $2(p+1)\cdot\mbox{ord}(B^2)$, where
$\mbox{ord}(n)$ is the multiplicative order of $n$.
In particular, $$k_{A,B}(p) \le 2(p+1)\cdot\mbox{ord}(B^2).$$
\end{thm}

The multiplicative order of $n$ is the smallest positive integer $t$ 
such that $n^t\equiv 1$ mod $p$.
The proof follows the proof of Theorem \ref{thm:23mod5} by noting
$\lambda \bar{\lambda}=B$, 
and hence $\lambda^{2(p+1)} = (-B)^{p+1} = B^2$ by Fermat's theorem.

Note that if $B=1$, then the original bound $2(p+1)$ still holds.  
For example, consider $E_{n+1}= 3 E_n + E_{n-1}$ mod $19$.
Then $A=3$, $B=1$, and $\Delta=13$.  
Since $13$ is a nonresidue mod $19$, 
our theorem shows $k_{3,1}(19)$ divides $40$ (and, in fact, is $40$).  
For the same sequence mod $11$, we find that $13$ is a nonresidue mod
$11$, so $k_{3,1}(11)$ divides $2(11+1)=24$
(and, in fact, is $8$).

For a general example where $B\neq 1$, consider
 $E_{n+1}= 3 E_n + 2 E_{n-1} \mod 7$.
Then $A=3$, $B=2$, and $\Delta=17$.  
Since $17$ is a nonresidue mod $7$, and $B^2=4$ satisfies $4^3 \equiv
1$ mod $7$, 
our theorem shows that the period $k_{3,2}$ divides $2(7+1)\cdot 3=48$ 
(and, in fact, is 48).
  
In general, we note that $ord(B^2)$ is at most $(p-1)/2$ by Fermat's
theorem, so 
the bound in Theorem \ref{Bsquaredbound} could be as high as
$2(p+1)(p-1)/2 = p^2-1$, the bound at the beginning of this paper.
This bound is actually achieved by $E_{n+1}= 3 E_n + 2 E_{n-1}$ mod
$37$.  This sequence begins
$$
0, 1, 3, 11, 2, 28, 14, 24, 26, 15, 23, 25, 10, 6, 1, 15, 10, 23, 15, 17, \dots
$$
and has period $1368=(37+1)(37-1)$, indicating that all
possible consecutive pairs other than $0, 0$ appear in this Fibonacci sequence mod $37$.



\end{document}